\documentclass[10pt,reqno]{amsart}
\usepackage[nodisplayskipstretch]{setspace}
\usepackage[brazilian,english]{babel} 
\usepackage[T1]{fontenc}
\usepackage{amsfonts,amsthm,amssymb,amsmath,mathrsfs}  
\usepackage{graphicx,color}       
\usepackage{latexsym}       
\usepackage{overpic}        
\usepackage[colorlinks=false, linktocpage=true]{hyperref}
\usepackage{graphicx}
\usepackage{multicol}
\usepackage{color}
\usepackage{multirow} 
\usepackage{mathtools}
\usepackage{xcolor}
\mathtoolsset{showonlyrefs}

\newtheorem{theorem}{Theorem}[section]

\newtheorem{lemma}[theorem]{Lemma}
\newtheorem{proposition}[theorem]{Proposition}

\theoremstyle{definition}

\newtheorem{remark}[theorem]{Remark}

\newcommand{\bu}{\bar{u}}

\newcommand{\R}{\mathbb{R}}

\newcommand{\ub}{\bar{u}}

\usepackage{a4wide}
\usepackage{mathabx}

\newcommand{\RE}{\mathrm{Re}}
\newcommand{\IM}{\mathrm{Im}}

\newcommand{\V}{\mathcal{V}}
\newcommand{\GR}{\Gamma_{\phi}}

\newcommand{\ec}{\frac{8-2b}{N-4}}

\newcommand{\n}{\Vert}
\newcommand{\nn}{\Vert_{L^2}}

\newcommand{\x}{|x|^{-b}}

\numberwithin{equation}{section}

\begin{document}
	\title[Inhomogeneous biharmonic NLS system]{BLOW-UP RESULTS FOR INHOMOGENEOUS FOURTH-ORDER NONLINEAR SCHRÖDINGER EQUATION}

	\author[M. Hespanha]{Maicon Hespanha}
	\author[R. Scarpelli]{Renzo Scarpelli}\address{ICEx, Universidade Federal de Minas Gerais, Av. Antônio Carlos, 6627, Caixa Postal 702, 30123-970 Belo Horizonte, MG, Brazil}
	\email{mshespanha@gmail.com}
	\email{renzoscb123@gmail.com}
	

	\begin{abstract}
		In this paper, we investigate the blow-up phenomenon of the $H^2$ norm of solutions to the inhomogeneous
		biharmonic Schrodinger equation in two distinct scenarios. First, we consider the case of negative energy, analyzing separately the cases of radial and non-radial solutions. Then, we examine the positive energy case, where the energy is below that of the ground state and
		the ``kinetic'' energy exceeds the corresponding value for the ground state, again distinguishing between radial and non-radial solutions. Our approach is based on convexity methods, employing virial identities in the analysis. 
		
	\end{abstract}
	
	\maketitle

	\section{introduction}
	We consider following initial value problem (IVP) for the focusing inhomogeneous biharmonic nonlinear Schrödinger IBNLS equation 
	\begin{equation}\label{IBNLS}
		\left\{\begin{array}{ll}
			iu_t + \Delta^2 u=|x|^{-b}|u|^{\frac{8-2b}{N-4}}u,\quad N\geq 5,\\
			u(0,x)=u_0\in H^2(\R^N),
		\end{array}\right.
	\end{equation}
	where $\Delta^2$ stands for the biharmonic operator, $ b>0$ and $u=u(t,x)$ is a complex-valued function in space-time $\R\times\R^N$. Note that when $b=0$ we get the classic biharmonic nonlinear Schrödinger equation BNLS, which was introduced by \cite{Karpman} and \cite{Karpmn2} and it has been studied throughout the past two decades, where was proved several results about local and global well-posedness, scattering and blowup both in the energy space $H^2$ and in the fractional Sobolev spaces $H^s$.  See, for example, Artzi-Koch-Saut \cite{AKS}, Boulenger-Lenzmann \cite{BL17}, Dihn \cite{D}, Miao-Xu-Zhao \cite{MXZ, MXZ2}, Pausader \cite{Pausader, Pausader2}, and the references therein.
	
	As for the classical NLS equation, the equation IBNLS \eqref{IBNLS} enjoys the conservation of mass and energy, respectively given by
	\begin{equation}\label{mass}
		M(u(t)):=\int|u(t,x)|^2dx=M(u_0),
	\end{equation}
	\begin{equation}\label{energy}
		E(u(t)):=\frac{1}{2}\int |\Delta u(t,x)|^2dx-\frac{N-4}{2(N-b)}\int |x|^{-b}|u(t,x)|^{\frac{2(N-b)}{N-4}}dx=E(u_0).
	\end{equation}
	In addition, the term ``energy-critical'' comes from the fact that  the scaling symmetry $u_\lambda(t,x)=\lambda^{\frac{(N-4)}{2}}u(\lambda^4t,\lambda x)$, $\lambda>0$, leaves   the energy invariant in $H^2$, the so-called energy space.

	In recent years, there has been increasing interest in the IBNLS equation \eqref{IBNLS}. Several results have been established regarding global and local well-posedness and scattering of solutions for the subcritical, mass-critical and intercritical cases. The reader can check An-Ryu-Kim \cite{ARK1}-\cite{ARK3}, Campos-Guzmán \cite{CG}, Cardoso-Guzmán-Pastor \cite{CGP}, Guzmán-Pastor \cite{GP1, GP2}, and the references therein to details. As far as we are aware, the study of well-posedness for the biharmonic nonlinear Schrödinger equation with inhomogeneous nonlinearities in the critical regime was first established in Guzmán-Pastor \cite{GP2}, where the authors proved results concerning well-posedness and stability with initial data in the critical space $\dot{H}^2$. More precisely, they proved local well-posedness in $\dot{H}^2$ for \eqref{IBNLS} with the restrictions $5\geq N\leq 11$ and $0<b<\frac{12-N}{N-2}$.  Then, in An-Ryu-Kim \cite{ARK3}, the authors used Sobolev-Lorentz space to guarantee that \eqref{IBNLS} is local well-posed for $N\leq 5$ and $0<b<\min\{4,N/2\}$. They also proved global well-posedness and scattering for small data under the same restrictions.

	Before describing our result, we recall that for the classical nonlinear Schrodinger equation NLS in the mass-critical regime, Glassey \cite{G} proved finite-time blow-up for negative energy $H^1$ solutions in the virial space. Later, Ogawa-Tsutsumi \cite{OT} removed the virial assumption in the radial case. For the 	inhomogeneous model, Cardoso-Farah \cite{CF} obtained the same result without assuming radial symmetry, exploiting the decay away from the origin together with a subcritical Gagliardo–Nirenberg inequality. In the context of the biharmonic NLS , Boulenger-Lenzmann \cite{BL17} established analogous blowup results in the radial setting.
	
	When it comes to IBNLS, to the best of our knowledge, Bai-Majdoub-Saanouni \cite{BMS} proved in the mass-critical regime that, for $N\geq 1$, $0<b<\min\{N/2,4\}$ and negative energy, that the solution of \eqref{IBNLS} either blows-up in finite or infinite time, in the senses that $\n \Delta u(t)\nn \gtrsim t^2$ for $t\gg 1$.

	Inspired by these works, we investigate the blow-up phenomenon for the inhomogeneous BNLS on the  energy-critical scenario. Our main results are stated in the following	theorems. The first one concerns to the radial case and reads as follows.

	\begin{theorem}\label{radial}
		Let $N\geq 5$, $0<b<\min\{4, N/2\}$ and $u_0\in H^2(\R^N)$ radial be such that $E(u_0)<0$ or $0\leq E(u_0)<E(W)$ and $\n \Delta u_0 \nn > \n \Delta W \nn$. Then $u(t)$ blows-up in finite time in the sense that exists $T^* >0$ such that 
		\[
		\lim_{t \rightarrow T^*} \n \Delta u(t) \nn = +\infty,
		\]
			where $W$ is the ground state solution to \eqref{IBNLS}.
		\end{theorem}
	For the non radial case we have the following result.
	\begin{theorem}\label{nao_radial}
		Let $N\geq 5$, $0< b < \min\{4, N/2\}$ and $u_0 \in H^2(\mathbb{R}^N)$. Then we have the following results:
		\begin{enumerate}
			\item If $0<b < 16/N$ and $E(u_0) <0$ then the corresponding maximal solution of \eqref{IBNLS} $u(t)$ either blows-up in finite time or in infinite time in the sense that
			\[
			\lim_{t \rightarrow T_{max}} \n \Delta u(t)\nn = + \infty .
			\]
			\item If $b \geq 16/N$ and $E(u_0)<0$ or $0\leq E(u_0)<E(W)$ with $\n \Delta u_0 \nn > \n \Delta W \nn$, then $u(t)$ blows-up in finite time in the sense that exists $T^* >0$ such that 
			\[
			\lim_{t \rightarrow T^*} \n \Delta u(t) \nn = +\infty,
			\]
			where $W$ is the ground state solution to \eqref{IBNLS}.
		\end{enumerate}
	\end{theorem}
	
	\begin{remark}
		The assumption $b\geq 16/N$ in the item (2) of Theorem \ref{nao_radial} is a technical restriction that we found in our argument. In order to handle the error term that appears in a virial-type inequality, we use classical Gagliardo-Nirenberg inequality combined with the free decay away from the origin that led us to the constraint $\frac{N(4-b)}{2(N-4)}\leq 2$, equivalently, $b\geq 16/N$ (see inequality \eqref{restriction} bellow). In an attempt to avoid this restriction, we employed a more refined estimate from the seminal paper by Caffarelli–Kohn–Nirenberg~\cite{CKN} to perform an interpolation, but ultimately arrived at the same constraint.
	\end{remark}
	
	The paper is organized as follows. In Section \ref{sec2} we introduce some notations and establishes some preliminary results that will be used throughout the work such as ground states solutions existence, virial-type inequality among others. In Section \ref{sec3} we prove Theorem \ref{radial} by introducing a virial-type inequality and using ODE argument. Finally, in Section \ref{sec4}, we prove the Theorem \ref{nao_radial}, where we benefit from the decay away from the origin that the term $|x|^{-b}$ gives to us.  
	
	\section{Notation and preliminary results}\label{sec2}
	\subsection{Notation} 
	Throughout the work we will use the standard notation in PDEs. In particular, $C$ will represent a generic constant which may vary from inequality to inequality. Given two positive number $a$ and $b$, we write $a\lesssim b$ whenever  $a\leq Cb$ for some constant $C>0$. Also, the subscript $a\lesssim_M b$ means that the constant $C$ on the assertion $a\leq Cb$ depends on the parameter $M$; similar for the case $a\gtrsim b$. For any $\theta>0$ we use the notation $O(R^{-\theta})$ to denote an infinitely small quantity of the same order as $R^{-\theta}$, i. e., $\frac{O(R^{-\theta)}}{R^{-\theta}}\rightarrow \ell\neq 0$ as $R\rightarrow\infty$.  For a complex number $z\in \mathbb{C}$, $\RE\,z$ and $\IM\,z$ represents its real and imaginary parts.  Also, $\bar{z}$ denotes the complex conjugate of $z$.
	
	We denote the standard Sobolev and the Lebesgue spaces by,  respectively, $H^2(\Omega)$ and $L^2(\Omega)$ with their usual norms. When $\Omega=\R^N$, we will abbreviate $H^2(\Omega)$ as $H^2$ and $L^2(\Omega)$ as $L^2$. 
	Integration of a function $f$ on $\R^N$ will be denoted by $\int fdx$ or simply $\int f$, if this will not cause a confusion.
	\subsection{Basic estimates}
	
	In this section we will set some estimates that will be useful later on this work. First, as a direct consequence of the Parseval's identity for the Fourier transform (see \cite[Proposition 2.1]{BMS}), we get
	\begin{equation}\label{gradest}
		\n \nabla u\nn ^2\lesssim \n \Delta u \nn^{1/2}\n u \nn^{1/2}.
	\end{equation}
	Inequality \eqref{gradest} generally fails to hold in bounded domains $\Omega$ except when $u\in W_0^{2,2}(\Omega)$.  Whether inequality \eqref{gradest} holds on exterior domains $\Omega$  without imposing zero trace on $\partial\Omega$ was studied in Crispo-Maremonti \cite{CM}. To summarize, the inequality holds when the set has the cone property, Therefore, since the exterior of an Euclidian ball in $\R^N$ possesses the cone property, it follows that 
	\begin{equation}\label{gradestball}
		\n \nabla u\n_{L^2(|x|>R)}\lesssim \n \Delta u\n_{L^2(|x|>R)}^{1/2}\n u \n_{L^2(|x|>R)}^{1/2}, \quad R>0.
	\end{equation}
	Moreover, a scaling argument shows that the constant $C$ in \eqref{gradestball} is independent of $R$. Another  inequality that plays an important role in the argument for the radial case is the well-known Strauss's inequality (see Strauss \cite{Strauss}), that reads as follows.
	\begin{equation}\label{strauss}
		\n u\n_{L^\infty(|x|>R)}\leq CR^{-(N-1)/2}\n u\n_{L^2}^{1/2}\n \nabla u\n_{L^2}^{1/2}.
	\end{equation}

	\subsection{Localized Virial Identity}
	
	Let $\phi:\R^N\rightarrow\R$ be a cut-off function. We define the virial quantity associated to a $H^2$ solution of \eqref{IBNLS} by
	\begin{equation}\label{virial}
		\V_\phi(t):= -2\IM\int\nabla\phi\cdot\nabla u(t) \ub(t)dx.	
	\end{equation}
	\begin{remark}
		The minus sign in the definition of the Virial identity, although uncommon, arises from the choice of a positive sign for the biharmonic term and a negative sign for the nonlinearity in equation~\eqref{IBNLS}. 
	\end{remark}
	\begin{lemma} \label{virial_lemma}
		Let $d\geq 5$. Suppose that $u\in C([0,T);H^2])$ is a solution of \eqref{IBNLS}. Then, for any $t\in[0,T)$, we have
		\[
		\begin{split}
			\V_\phi'(t)&=-4\sum_{j,k}^N\int \partial_{jk}\Delta \phi \partial_j u\partial_k\ub dx +\int \Delta^3\phi |u|^2dx\\
			&\quad+8\sum_{i,j,k}^N\int \partial_{jk} \phi \partial_{ik} u \partial_{ij}\ub dx-2\int\Delta^2\phi |\nabla u|^2dx\\
			&\quad-\frac{8-2b}{N-b}\int \Delta\phi \x|u|^{\ec +2}+\frac{2(N-4)}{N-b}\int \nabla\phi\cdot\nabla(\x)|u|^{\frac{2(N-b)}{N-4}}\\
		\end{split}
		\]
	\end{lemma}
	\begin{proof}
		We start by noticing that
		\begin{equation}\label{R1}
			i\V_\phi(t)=-2i\IM \int \bu\nabla u\cdot\nabla \phi dx
			=\int u\nabla \bu\cdot\nabla \phi dx- \int  \bu\nabla u\cdot \phi dx.
		\end{equation}
		On the other hand, integration by parts yields
		\begin{equation}\label{R2}
			\int|u|^2\Delta\phi dx =-\int u\nabla \bu\cdot\phi dx-\int \bu\nabla u\cdot\nabla \phi dx.
		\end{equation}
		Combining \eqref{R1} and \eqref{R2} we get
		\begin{equation}\label{R3}
			\begin{split}
				i	\V_\phi(t)&=-2\int\bu\nabla u\cdot\nabla\phi dx-\int\bu\Delta\phi  u dx\\
				&=-\int \bu[\nabla\phi\cdot\nabla +\Delta\phi]u dx-\int \bu\nabla u\cdot\nabla\phi dx.
			\end{split}	
		\end{equation}
		If we rewrite $-\int \bu \nabla u \cdot \nabla\phi :=a+bi$, we conclude that $\V_\phi(t) = 2b$. Therefore
		\[
		\V_\phi(t)=-2\RE \int \bu\Gamma_{\phi}u dx,
		\]
		where $\GR:=-i\left(\nabla\phi \cdot\nabla + \Delta\phi \right)$. By introducing the notation $\langle u, v \rangle = 2 \RE (u,v)_{L^2}$ we get
		\[ \V_\phi(t):= - \langle  u ,\GR u \rangle=-\langle i  u, i\GR u\rangle,  \]
		taking the time derivative and using \eqref{IBNLS}, we obtain
		\[\begin{split}
			\V'_\phi(t)&=  \langle u ,\left[\Delta^2, i\GR \right]u  \rangle -  \langle u , [|x|^{-b}|u|^{\frac{8-2b}{N-4}}u, i\GR] u  \rangle\\
			&:= A-B,
		\end{split}\]
		where $[X,Y]=XY-YX$. Note that we may treat the term $A$ in the same way as $\mathcal{A}_R^{[1]}$ in \cite[Lemma 3.1]{BL17} (so we will omit the details) to get
		\begin{equation}\label{A1}
			\begin{split}
				A&= 8\RE \sum_{i,j,k}\int\partial^2_{j,k}\phi \partial^2_{i,k}u \partial^2_{i,j}\bu dx\\
				&\quad-4\RE \sum_{j,k}\int\partial_j\bu \partial^2_{j,k}\Delta\phi \partial_ku dx\\
				&\quad-2 \int\Delta^2\phi|\nabla u |^2dx+  \int\Delta^3\phi |u|^2dx.\\
			\end{split}
		\end{equation}
		For $B$, we use integration by parts 
		\begin{equation}\label{B1}
			\begin{split}
				B&= 2 \RE\left[\int (|x|^{-b}|u|^{\frac{8-2b}{N-4}}u) i\GR\bu dx-\int\bu  i\GR |x|^{-b}|u|^{\frac{8-2b}{N-4}}u dx\right]\\
				&=2 \RE\left[ \int |x|^{-b}|u|^{\frac{8-2b}{N-4}}u \nabla\phi \cdot\nabla\bu dx+\int |x|^{-b}|u|^{\frac{2(N-b)}{N-4}} \Delta\phi dx \right]\\
				&\quad -2 \RE\left[ \int \bu \nabla\phi_R\cdot\nabla(|x|^{-b}|u|^{\frac{8-2b}{N-4}}u)dx +\int |x|^{-b}|u|^{\frac{2(N-b)}{N-4}}\Delta\phi dx\right]\\
				&=2 \RE\left[2\int |x|^{-b}|u|^{\frac{8-2b}{N-4}}u\nabla \bu \cdot\nabla\phi dx+\int (|x|^{-b}|u|^{\frac{2(N-b)}{N-4}}) \Delta\phi dx\right].\\
			\end{split}
		\end{equation}
		For the first one, we use the equality $\nabla (|u|^{\frac{2(N-b)}{N-4}})=\frac{2(N-b)}{N-4}|u|^{\frac{8-2b}{N-4}}\RE\nabla\bu u$ and integration by parts to get
		\begin{equation}
			\begin{split}
				4 \RE\int |x|^{-b}|u|^{\frac{8-2b}{N-4}}u\nabla \bu \cdot\nabla\phi dx &= -\frac{2(N-4)}{N-b}\int \nabla(|x|^{-b})\nabla\phi |u|^{\frac{2(N-b)}{N-4}}dx \\
				&-\frac{2(N-4)}{N-b}\int (|x|^{-b}|u|^{\frac{2(N-b)}{N-4}}) \Delta\phi dx.
			\end{split}
		\end{equation}
		Combining the above equality with \eqref{B1} we get the desired result. 
	\end{proof}
	For $R>0$, set the radial cut-off function
	\begin{equation}\label{phiR}
		\phi_R(r)=R^2\phi\left(\frac{r}{R}\right),\quad r=|x|,
	\end{equation}
	where
	\[
	\phi(r)=\int_0^r\chi(s)ds,
	\]
	with $\chi:[0,\infty)\rightarrow[0,\infty)$ is a smooth function satisfying
	\[
	\chi(r)=\begin{cases} 
		2r&\hbox{if}\quad0\leq r\leq 1,\\
		2r-2(r-1)^2&\hbox{if}\quad 1<r<1+1/\sqrt{2}\\
		\chi'(r)<0&\hbox{if}\quad 1+1+1/\sqrt{2}<r<2\\
		0&\hbox{if} \quad r\geq 2.
	\end{cases}
	\]
	The next result summarize some properties of the cut-off function $\phi_R$. The proof can be found in \cite[Lemma 2.1]{D}.
	\begin{lemma}\label{l27}
		Let $\phi_R$ be the as in \eqref{phiR}. Then
		\begin{equation}\label{214}
			\phi'_R(r)=2r,\quad \phi''_R(r)=2,\quad \hbox{for } r\leq R,
		\end{equation}
		\begin{equation}\label{215}
			\phi'_R(r)-r\phi''_R(r)\geq 0,\quad \phi'_R(r)\leq 2r,\quad \phi''_R(r)\leq 2,		
		\end{equation}
		\begin{equation}\label{216}
			\n \nabla^j \phi_R\n_{L^\infty}\lesssim R^{2-j}, \quad 0\leq j\leq 6,
		\end{equation}
		and
		\begin{equation}\label{217}
			supp(\nabla^j \phi_R)\subset \begin{cases}
				\left\{|x|\leq 2R\right\} &\hbox{if } j=1,2,\\
				\left\{R\leq |x|\leq 2R  \right\}&\hbox{if } 3\leq j\leq 6.
			\end{cases}
		\end{equation}
	\end{lemma}

	\subsection{Ground State}
	
	Based on the results of Caffarelli-Kohn-Nirenberg \cite{CKN}, the following inequality is known to hold.
	
	\begin{theorem}{[Critical Sobolev inequality]}
		Let $u \in \dot{H}^2(\mathbb{R}^N)$ with $N >4$ and $b < min\{4, N/2\}$. Then,
		\begin{equation}\label{GN}
			\int |x|^{-b} |u|^{ \frac{2(N-b)}{N-4}}dx \leq K_{opt} \n \Delta u \nn^{\frac{2(N-b)}{N-4}}.
		\end{equation}
	\end{theorem}
		Introducing the ``Weinstein function'' as
		\[
		J(f):= \frac{\n \Delta u \nn^{\frac{2(N-b)}{N-4}}}{\int |x|^{-b} |u|^{ \frac{2(N-b)}{N-4}}dx},
		\]
				Following the standard convention, $0\neq W\in {H}^2(\R^N)$ is said to be ground state if optimizes \eqref{GN}, i.e.
				\[J(W)=\inf\{J(f),\, f\in \dot{H}^2,\, f\neq0\}\]
				In Bhakta-Musina \cite[Theorem 5.1]{BM}, was proved that there is a unique, decreasing and positive radial ground state solution, namely $W$, which also satisfy the elliptic equation
	\begin{equation}\label{eliptica}
		\Delta^2 W(x) = |x|^{-b}|W(x)|^{\frac{8-2b}{N-4}}W(x).
	\end{equation}

	The following lemma establishes some Pohozaev identities for $W$ and characterizes the optimal constant to \eqref{GN}. 
	\begin{lemma}\label{pohozaev}[Pohozaev-type identities] Let $W\in \dot{H}^2\cap L^p$, with $p=\frac{2(N-b)}{N-4}$, be the unique solution to \eqref{eliptica}. Then the following identities holds
		\begin{equation}\label{P1} 
			\n \Delta W\nn^2=\displaystyle\int |x|^{-b} |W|^{p}dx;\end{equation}
		\begin{equation}\label{P2} 
			E(W)= \frac{4-b}{2(N-b)} \n \Delta W \nn^2;\end{equation}
		\begin{equation}\label{optimal}
			K_{opt} = \displaystyle\frac{1}{\n \Delta W \nn^{\frac{8-2b}{N-4}}}.\end{equation}
			\end{lemma}
	\begin{proof}
		We get the first one by multiplying \eqref{eliptica} by $W$, integrating and using integration by parts. The second one is a direct consequence of \eqref{P1} and the definition of the energy \eqref{energy}. Finaly, since $W$ is the ground state solution,  it attains the optimal value in \eqref{GN}, leading us to
		\[
		\int |x|^{-b} |W|^{p}dx = K_{opt} \n \Delta W \nn^{p}.
		\]
		Combining with \eqref{P1} we get \eqref{optimal}.
	\end{proof}
	A direct consequence of the Pohozaev-type identities is the following lemma.
	\begin{lemma} \label{groundstate}
		Let $u(t)$ be a solution to \ref{IBNLS} satisfying $0 <E(u_0) < E(W)$ and $\n \Delta u_0\nn > \n \Delta W\nn$. Then, there exists $\delta >0$ such that 
		\[
		E(u(t)) \leq (1-\delta)\frac{4-b}{2(N-b)} \n \Delta u(t) \nn^2.
		\]
	\end{lemma}
	\begin{proof}
		Given the initial conditions $0 <E(u_0) < E(W)$ and $\n \Delta u_0\nn > \n \Delta W\nn$, a standard continuity argument implies that $\n \Delta u(t)\nn > \n \Delta W\nn$ for all $t$ in the maximal interval of existence. Indeed, by the optimal Gagliardo–Nirenberg inequality \eqref{GN}, we have
		\begin{align} \label{rl1}
			E(u(t)) &= \frac{1}{2} \n \Delta u(t)\nn^2 - \frac{N-4}{2(N-b)} \int |x|^{-b} |u(t)|^{2 + \frac{8-2b}{N-4}} dx \nonumber \\
			&\geq \frac{1}{2} \n \Delta u(t)\nn^2 - \frac{N-4}{2(N-b)} \frac{ \n \Delta u(t)\nn^{2 + \frac{8-2b}{N-4}}}{\n \Delta W \nn^{\frac{8-2b}{N-4}}}.
		\end{align}
		
		Now, define the function
		\[
		f(x) = \frac{1}{2}x^2 - \frac{N-4}{2(N-b)} \frac{ x^{2 + \frac{8-2b}{N-4} }}{\n \Delta W \nn^{\frac{8-2b}{N-4}}}
		\]
		and note that $f(x)$ attains its maximum at $x_m = \n \Delta W \nn$, with $f(x_m) = E(W)$. Since $E(u_0) < E(W)$ and $\n \Delta u_0\nn > \n \Delta W \nn$, it follows from \eqref{rl1} and the continuity of the $\dot{H}^2$ norm of the solution that $\n \Delta u(t)\nn > \n \Delta W\nn$.
		
		We also note, by Pohozaev's identity \eqref{P2} that
		\[
		E(W) = \frac{4-b}{2(N-b)} \n \Delta W \nn^2.
		\]
		Therefore, using that $E(u(t)) < E(W)$, there exists $\delta >0$ such that $E(u(t)) < (1-\delta)E(W)$. Finally, using that  $\n \Delta u(t)\nn > \n \Delta W\nn$ we get the result.
	\end{proof}
	
	\section{Proof of Theorem \ref{radial}}\label{sec3}
	We start by setting the Morawetz potential 
	\begin{equation}\label{morawetz}
		\V_R(t):=\V_{\phi_R}(t)=-2\IM\int \nabla \phi_R\cdot \nabla u(t)\ub(t)dx,
	\end{equation}
	where $\phi_R$ is defined as in \eqref{phiR}.
	
	\begin{proposition}
		Let $u(t)$ be a radial solution of \eqref{IBNLS}.  We have \begin{equation} \label{cota_virial}
			\begin{split}
				\V_R'(t) &\leq \frac{32(N-b)}{(N-4)}E(u_0) - \frac{16(4-b)}{(N-4)} \n \Delta u(t) \nn^2 \\ &+ O\left(R^{-4} + R^{-2}\n \nabla u(t)\nn^2 + R^{-b -\frac{(N-1)(4-b)}{N-4}} \n \nabla u(t) \nn^{\frac{4-b}{N-4}}\right).
			\end{split}
		\end{equation}
	\end{proposition}
	\begin{proof}
		We begin by recalling the following radial identities
		\begin{equation}\label{radial grad}
			\nabla_x=\frac{x}{r}\partial_r,
		\end{equation}
		\begin{equation}\label{radial lap}
			\Delta_x=\partial_r^2+\frac{N-1}{r}\partial_r,
		\end{equation} 
		\[
		\partial_j\partial_k=\left( \frac{\delta_{jk}}{r}-\partial{x_jx_k}{r^3} \right)\partial_r+\frac{x_jx_k}{r^2}\partial_r^2,
		\]
		\[
		\nabla_x\cdot \nabla(|x|^{-b})=-b|x|^{-b}\frac{\partial_r}{r}.
		\]
		Applying above identities in Lemma \ref{virial_lemma}, by using the definition of $\phi_R$ \eqref{phiR}, and the conservation of mass, we obtain
		\[
		\begin{split}
			\V_r'(t) &\leq 16 \n \Delta u(t) \nn^2 - \frac{8 - 2b}{N - b} \int \Delta \phi_R\, |x|^{-b} |u|^{\frac{8 - 2b}{N - 4} + 2}\, dx \\
			&\quad - \frac{2b(N - 4)}{N - b} \int \frac{\nabla \phi_R \cdot x}{|x|^2}\, |x|^{-b} |u|^{\frac{8 - 2b}{N - 4} + 2}\, dx + O(R^{-4} + R^{-2} \n \nabla u(t) \nn^2).
		\end{split}
		\]
		
		Using \eqref{radial grad} we deduce that $\frac{\nabla \phi_R \cdot x}{|x|^2} = \frac{\phi_R'}{r}$. Combining this with \eqref{radial lap} we deduce
		\[
		\begin{split}
			\V_R'(t) &\leq 16 \n \Delta u(t) \nn^2 - \int \Phi_R(r)\, |x|^{-b} |u|^{\frac{8 - 2b}{N - 4} + 2}\, dx + O(R^{-4} + R^{-2} \n \nabla u(t) \nn^2),
		\end{split}
		\]
		where
		\[
		\Phi_R(r) = \frac{8 - 2b}{N - b} \phi_R'' + \left( \frac{8 - 2b}{N - b}(N - 1) + \frac{2b(N - 4)}{N - b} \right) \frac{\phi_R'}{r}.
		\]
		
		A direct computation shows that inside the ball of radius $R > 0$, we have $\Phi_R(r) = 16$. Therefore,
		\[
		\begin{split}
			\V_R'(t) &\leq 16 \n \Delta u(t) \nn^2 - \int (\Phi_R(r) - 16)\, |x|^{-b} |u|^{\frac{8 - 2b}{N - 4} + 2}\, dx \\
			&\quad - 16 \int |x|^{-b} |u|^{\frac{8 - 2b}{N - 4} + 2}\, dx + O(R^{-4} + R^{-2} \n \nabla u(t) \nn^2).
		\end{split}
		\]
		
		Using the conservation of energy, we obtain
		\begin{equation} \label{full_virial}
			\begin{split}
				\V_R'(t) &\leq \frac{32(N - b)}{N - 4} E(u_0) - \frac{16(4 - b)}{N - 4} \n \Delta u(t) \nn^2 \\
				&\quad - E_r + O(R^{-4} + R^{-2} \n \nabla u(t) \nn^2),
			\end{split}
		\end{equation}

		where
		\[
		E_r = \int_{|x| > R} (\Phi_R(r) - 16)\, |x|^{-b} |u|^{\frac{8 - 2b}{N - 4} + 2}\, dx.
		\]
		
		To estimate this error term, we apply Strauss’s inequality \eqref{strauss} and the conservation of mass to get
		\[
		\begin{split}
			|E_r| &\leq R^{-b} M(u_0)\, \| u(t) \|_{L^\infty(|x| > R)}^{\frac{8 - 2b}{N - 4}} \\
			&\lesssim R^{-b - \frac{(N - 1)(4 - b)}{N - 4}} \n \nabla u(t) \nn^{\frac{4 - b}{N - 4}},
		\end{split}
		\]
		which concludes the proof of the lemma. 
	\end{proof}
	
	With this lemma in hand, we can now establish the blow-up result for the radial case. 
	
	\begin{proof}[Proof of Theorem \ref{radial}]
		We begin by showing the result in the case of negative energy. Using \eqref{gradest}, Young’s inequality and the conservation of mass, we observe that
		\[
		\begin{split}
			\n \nabla u(t) \nn^2 &\leq M(u_0)^{1/2} \n \Delta u(t) \nn \\
			&\lesssim 1 + \n \Delta u(t) \nn^2,
		\end{split}
		\]
		and similarly,
		\[
		\n \nabla u(t) \nn^{\frac{4 - b}{N - 4}} \lesssim 1 + \n \Delta u(t) \nn^2.
		\]
		
		Therefore, by choosing $R > 0$ sufficiently large and absorbing the error terms, we obtain
		\begin{equation} \label{virial_final}
			\V_R'(t) \leq \frac{16(N - b)}{N - 4} E(u_0) - \frac{8(4 - b)}{N - 4} \n \Delta u(t) \nn^2.
		\end{equation}
		
		Using the hypothesis $E(u_0) <0$ we get that $\V_R'(t)$ is strictly negative.  We now apply a standard ODE argument to conclude the proof. Suppose, for contradiction, that the maximal time of existence is $T = +\infty$. Since $\V_R'(t)$ is strictly negative, there exists a time $t_0$ such that $\V_R(t) \leq 0$ for all $t \geq t_0$. In particular, we have
		\[
		\V_R(t) \leq -C \int_{t_0}^t \| \Delta u(s) \|^2\, ds,
		\]
		for some constant $C > 0$. Moreover, by definition, we also have the estimate $|\V_R(t)| \lesssim \n \Delta u(t) \nn^{1/2}$. Combining these facts yields
		\begin{equation}
			\label{cota1}
			\V_R(t) \leq -C \int_{t_0}^t |\V_R(s)|^4\, ds,
		\end{equation}
		where $C = C(u_0, R)$. 
		
		Define
		\[
		A(t) = \int_{t_0}^t |\V_R(s)|^4\, ds.
		\]
		Note that $A(t)$ is non-decreasing, and from \eqref{cota1} we obtain
		\[
		A'(t) = |\V_R(t)|^4 \geq C^4 A(t)^4.
		\]
				Integrating this inequality from $t_1 > t_0$ to $t$ yields
		\[
		\int_{t_1}^t \frac{A'(s)}{A(s)^4}\, ds \geq C^4(t - t_1),
		\]
		which implies
		\[
		A(t)^3 \geq \frac{A(t_1)^3}{1 - 3C^4 A(t_1)^3 (t - t_1)}.
		\]
				Hence, $A(t) \to +\infty$ as $t \to t^* = t_1 + \frac{1}{3C^4 A(t_1)^3}$. Since $\V_R(t) \leq -C A(t)$, it follows that $\V_R(t) \to -\infty$ in finite time, which implies that $\n \Delta u(t) \nn \to \infty$ in finite time. This contradicts the assumption that the solution exists globally in time.
		
		We now turn our attention to the case of positive energy. Using the estimate \eqref{virial_final} in combination with Lemma \ref{groundstate}, we obtain
		\[
		\begin{split}
			\V_R'(t) &\leq \frac{16(N - b)}{N - 4} E(u_0) - \frac{8(4 - b)}{N - 4} \n \Delta u(t) \nn^2 \\
			&\leq (1 - \delta)\frac{8(4 - b)}{N - 4} \n \Delta u(t) \nn^2 - \frac{8(4 - b)}{N - 4} \n \Delta u(t) \nn^2 \\
			&= -8\delta \frac{4 - b}{N - 4} \n \Delta u(t) \nn^2,\\
			&\leq -8\delta \frac{4 - b}{N - 4} \n \Delta W \nn^2 \\
			&< 0\\
		\end{split}
		\]
		the argument proceeds along the same lines.
		
	\end{proof}
	
	\section{Proof of Theorem \ref{nao_radial}}\label{sec4}
	
	The main difference in the proof of this result lies in how we handle the error term
	\[
	E_R = \int_{|x| > R} (\Phi_R(r) - 16)\, |x|^{-b} |u|^{\frac{8 - 2b}{N - 4} + 2}\, dx,
	\]
	which appears in \eqref{full_virial}. Since the solution is no longer radial, we cannot apply Strauss' inequality. However, due to the presence of the $|x|^{-b}$ term, we benefit from decay away from the origin (see, e.g., Cardoso-Farah \cite{CF}).
	
	\begin{proof}[Proof of Theorem \ref{nao_radial}] 
		We split the proof into three cases.
		\begin{enumerate}
			\item \textbf{Case $0 < b < 16/N$}.
			
			Suppose that the solution is global in time and satisfies $\| \Delta u(t) \| \leq M$ for all $t$. Then, we can estimate the error term using the classical Gagliardo–Nirenberg inequality as follows:
			\[
			\begin{split}
				|E_r| &\lesssim R^{-b} \int |u|^{\frac{8 - 2b}{N - 4} + 2}\, dx \\
				&\lesssim R^{-b} \n \Delta u(t) \nn^{\frac{N(4 - b)}{2(N - 4)}} \| u(t) \nn^{\frac{b}{2}} \\
				&\lesssim_{u_0, M} R^{-b}.
			\end{split}
			\]
			Combining this with \eqref{full_virial}, and choosing $R > 0$ sufficiently large, we obtain
			\[
			\V_R'(t) \leq \frac{8(N - b)}{N - 4} E(u_0),
			\]
			which implies $\V_R'(t) \leq -C < 0$. Using the Fundamental Theorem of Calculus and the bound 
            \[
            \begin{split}
                |\V_R(t)| &\leq \n \nabla u(t)\nn \n u_0\nn \\
                &\lesssim_{u_0} \n \Delta u(t)\nn^{1/2} \\
                &\lesssim_{u_0} M^{1/2}
            \end{split}
            \]
            we reach a contradiction.
			\\
			\item \textbf{Case $b \geq 16/N$ and $E(u_0) < 0$}.
			
			The key observation in this case is that for $b \geq 16/N$, we have $\frac{N(4 - b)}{2(N - 4)} \leq 2$. Therefore, by Gagliardo Nirenberg's inequality and conservation of mass we have
			\begin{equation} \label{restriction}
			\begin{split}
				|E_r| &\lesssim R^{-b} \int |u|^{\frac{8 - 2b}{N - 4} + 2}\, dx \\
				&\lesssim R^{-b} \n \Delta u(t) \nn^{\frac{N(4 - b)}{2(N - 4)}} \n u(t) \nn^{\frac{b}{2}} \\
				&\lesssim_{u_0} R^{-b} + R^{-b} \| \Delta u(t) \nn^2,
			\end{split}
			\end{equation}
			since $\| \Delta u(t) \nn^{\frac{N(4 - b)}{2(N - 4)}} \leq 1 + \| \Delta u(t) \nn^2$. Again, choosing $R > 0$ sufficiently large yields \eqref{virial_final}. The same ODE argument as in the proof of Theorem \ref{radial} then implies finite-time blow-up.
			\\
			\quad\\
			\textbf{Case $b \geq 16/N$ and $0 \leq E(u_0) < E(W)$ with $\| \Delta u_0 \nn > \| \Delta W \nn$}.
			
			Using the estimate
			\[
			|E_r| \lesssim_{u_0} R^{-b} + R^{-b} \| \Delta u(t) \nn^2,
			\]
			and choosing $R > 0$ large enough, we can combine this with Lemma \ref{groundstate} to obtain
			\[
			\V_R'(t) \leq -8\delta \frac{4 - b}{N - 4} \| \Delta u(t) \nn^2,
			\]
			and the argument then proceeds exactly as in the proof of Theorem \ref{radial}.
		\end{enumerate}
	\end{proof}
	
	\subsection*{Acknowledgment}
		M. H. is supported by Conselho Nacional de Desenvolvimento Científico e Tecnológico - CNPq. R.S. is supported by Fundação de Amparo a Pesquisa do Estado de Minas Gerais - FAPEMIG. The authors thank Professors Luiz Gustavo Farah and Luccas Campos for his support and valuable suggestions during the preparation of this work.

\end{document}